\documentclass[12pt]{article}
\usepackage{amssymb}
\usepackage{amsfonts}
\usepackage{amsmath,amssymb,amsthm,mathrsfs}
\usepackage{CJK,indentfirst,amsmath,amsfonts,amssymb,amsthm,cite}
\def\p1{\phi^{-1}}

\makeatletter
\renewcommand\section{\@startsection {section}{1}{\z@}%
                                    {-3.5ex \@plus -1ex \@minus -.2ex}%
                                    {2.3ex \@plus.2ex}%
                                    {\normalfont\small\bfseries}}

\theoremstyle{definition}
\newtheorem{theorem}{Theorem}[section]

\newtheorem{corollary}[theorem]{Corollary}

\begin{document}

\title {{\bf\large Nonlinear mixed Jordan triple $*$-derivations on $*$-algebras}
\footnotetext{ The authors are supported by the Natural Science Foundation of Shandong Province,
China (Grant No. ZR2018BA003) and the National Natural Science Foundation of China
(Grant No. 11801333).}\footnotetext{ $^{*}$ Corresponding author. E-mail: lcjbxh@163.com (C. Li).}}
\author {{Dongfang Zhang}, {Changjing Li$\ ^{* }$ }\\
{\small  School of Mathematics and Statistics, Shandong Normal University, Jinan 250014, P. R. China}\\
} \date{ } \maketitle

\begin{abstract}
Let $\mathcal {A}$ be a unital $\ast$-algebra. For $A, B\in\mathcal
{A}$, define by $[A, B]_{*}=AB-BA^{\ast}$ and $A\bullet B=AB+BA^{\ast}$ the new products of $A$ and $B$. In this paper, under some mild conditions on  $\mathcal {A}$, it is shown that a  map $\Phi:
\mathcal {A}\rightarrow \mathcal {A}$ satisfies
$\Phi([A\bullet B, C]_{*})=[\Phi(A)\bullet B, C]_{*}+[A\bullet \Phi(B), C]_{*}+[A\bullet B, \Phi(C)]_{*}$ for all
$A, B,C\in\mathcal {A}$ if and only if $\Phi$ is an additive
$*-$derivation.  In particular, we apply the above result to  prime $\ast$-algebras, von Neumann algebras with no central summands of type $I_1$, factor von Neumann algebras and standard operator algebras.

 \vskip 2mm


\end{abstract} \maketitle

\section{Introduction }
Let $\mathcal {A}$ be a $*$-algebra over the complex field
$\mathbb{C}$.  For $A, B\in\mathcal
{A}$, define the skew Lie product of $A$ and  $B$ by
$[A,B]_{\ast}=AB-BA^{\ast}$  and the Jordan $*$-product  of $A$ and  $B$ by $A\bullet B=AB+BA^{\ast}$. The skew Lie product and  the Jordan $*$-product are fairly meaningful and important in some research topics.
They were extensively studied because they naturally arise in the problem of representing
quadratic functionals with sesquilinear functionals (see \cite{12,13,14}) and in the problem of characterizing ideals
(see \cite{2,11}). Particular attention has been paid
to understanding maps which preserve the skew Lie product or  the Jordan $*$-product on $\ast$-algebras (see \cite{1,3,4,5,7,15,8,9,16}).

Recall that an additive map $\Phi: \mathcal
{A}\rightarrow\mathcal {A}$ is said to be an additive derivation if
$\Phi(AB)=\Phi(A)B+A\Phi(B)$ for all $A, B\in\mathcal {A}.$  Furthermore, $\Phi$ is said to be an additive $*$-derivation if it is an additive
derivation and satisfies $\Phi(A^{*})=\Phi(A)^{*}$ for all
$A\in\mathcal {A}.$ A not necessarily linear map
$\Phi:\mathcal {A}\rightarrow\mathcal {A}$ is said to be a nonlinear Jordan $*$-derivation or a nonlinear skew Lie derivation
if $$\Phi(A\bullet B)=\Phi(A)\bullet B+A\bullet \Phi(B)$$  or $$\Phi([A,B]_{\ast})=[\Phi(A), B]_{\ast}+[A, \Phi(B)]_{\ast}$$
for all $A, B\in\mathcal {A}$.
Yu and Zhang in \cite{Yu} proved that every nonlinear skew Lie derivation on factor von Neumann algebras
is an additive $*$-derivation.  Jing in \cite{Ji} studied nonlinear skew Lie derivations on standard operator algebras. Let $\mathcal {A}$ be a standard operator algebra on a complex Hilbert space $H$ which is closed under the adjoint operation. It was shown that every nonlinear skew Lie derivation $\Phi$ on $\mathcal {A}$ is automatically linear. Moreover, $\Phi$ is an inner $*$-derivation. Taghavi et al. \cite{Ta} and Zhang \cite{Zh} independently investigated nonlinear Jordan $*$-derivations on factor von Neumann algebras, respectively. It turns out that every nonlinear Jordan $*$-derivation between factor von Neumann algebras is an additive $*$-derivation. Li et al. in \cite{Li} investigated nonlinear skew Lie derivations and Jordan $*$-derivations on von Neumann algebras with no central summands of type $I_1$.

Given the consideration of nonlinear Jordan $*$-derivations and nonlinear skew Lie derivations, we can further develop them in one natural way. A map $\Phi:\mathcal {A}\rightarrow\mathcal {A}$ is said to be a nonlinear Jordan triple $*$-derivation or a nonlinear skew Lie triple derivation
if $$\Phi(A\bullet B\bullet C)=\Phi(A)\bullet B\bullet C+A\bullet \Phi(B)\bullet C+A\bullet B\bullet \Phi(C)$$
or
$$\Phi([[A,B]_{\ast}, C]_{\ast})=[[\Phi(A), B]_{\ast}, C]_{\ast}+[[A, \Phi(B)]_{\ast}, C]_{\ast}+[[A, B]_{\ast}, \Phi(C)]_{\ast}$$
for all
$A, B, C\in\mathcal {A}$. Li et al.\cite{Li2} proved that every nonlinear skew Lie triple derivation on factor von Neumann algebras
is an additive $*$-derivation.   Fu and An \cite{Fu} proved that $\Phi$ is a nonlinear skew Lie triple derivation on von Neumann algebras with no central summands of type $I_1$ if and only if $\Phi$ is an additive $*$-derivation.
Zhao and Li \cite{Zhao} proved that every nonlinear Jordan triple $*$-derivation between von Neumann algebras with no central summands of type $I_{1}$  is an additive $*$-derivation. Lin \cite{lin,lin2} studied the nonlinear  skew Lie $n$-derivations on standard operator algebras and von Neumann algebras with no central summands of type $I_1$.

In this paper, we will study the nonlinear mixed Jordan triple $*$-derivations on $*$-algebras. A map $\Phi:\mathcal {A}\rightarrow\mathcal {A}$ is said to be a nonlinear mixed Jordan triple $*$-derivation if $$\Phi([A\bullet B, C]_{*})=[\Phi(A)\bullet B, C]_{*}+[A\bullet \Phi(B), C]_{*}+[A\bullet B, \Phi(C)]_{*}$$ for all $A, B,C\in\mathcal {A}$.  Under some mild conditions on  a $*$-algebra $\mathcal {A}$, we prove that a  map $\Phi:
\mathcal {A}\rightarrow \mathcal {A}$ is a nonlinear mixed Jordan triple $*$-derivation if and only if $\Phi$ is an additive
$*-$derivation.  In particular, we apply the above result to  prime $\ast$-algebras, von Neumann algebras with no central summands of type $I_1$, factor von Neumann algebras and standard operator algebras.

\section{The main result and its proof}

Our main result in this paper reads as follows.
\begin{theorem}
Let $\mathcal {A}$ be a  unital $\ast$-algebra with the unit $I$. Assume that $\mathcal {A}$ contains a nontrivial projection $P$ which satisfies
$$(\spadesuit)~~~X\mathcal {A}P=0~~~implies~~~X=0$$
and
$$(\clubsuit)~~~X\mathcal {A}(I-P)=0~~~implies~~~X=0.$$ Then a  map $\Phi:
\mathcal {A}\rightarrow \mathcal {A}$ satisfies
$\Phi([A\bullet B, C]_{*})=[\Phi(A)\bullet B, C]_{*}+[A\bullet \Phi(B), C]_{*}+[A\bullet B, \Phi(C)]_{*}$ for all
$A, B,C\in\mathcal {A}$ if and only if $\Phi$ is an additive
$*-$derivation.
\end{theorem}

\begin{proof} Let $P_{1}=P$ and
$P_{2}=I-P.$ Denote $\mathcal {A}_{jk}=P_{j}\mathcal {A}P_{k}, j, k=1,2.$ Then $\mathcal
{A}=\sum^{2}_{j,k=1}\mathcal {A}_{jk}.$ In all that
follows, when we write $A_{jk},$ it indicates that
$A_{jk}\in\mathcal {A}_{jk}$. Clearly, we only need to prove the necessity. We will complete the proof by
several claims. \\
\textbf{Claim 1}. $\Phi(0)=0.$

Indeed, we have $$\Phi(0)=\Phi([0\bullet 0, 0]_{*})=[\Phi(0)\bullet 0, 0]_{*}+[0\bullet \Phi(0), 0]_{*}+[0\bullet 0, \Phi(0)]_{*}=0.$$
\textbf{Claim 2}. $\Phi$ is additive.

We will complete the proof of Claim 2 by proving several steps.\\
\textbf{ Step 2.1}. For every $A_{12}\in\mathcal {A}_{12},
B_{21}\in\mathcal {A}_{21},$ we have
$$\Phi(A_{12}+B_{21})=\Phi(A_{12})+\Phi(B_{21}).$$

We only need show that
$T=\Phi(A_{12}+B_{21})-\Phi(A_{12})-\Phi(B_{21})=0.$
Since $$[I\bullet (i(P_{2}-P_{1})),A_{12}]_{*}=[I\bullet (i(P_{2}-P_{1})),B_{21}]_{*}=0,$$
where $i$ is the imaginary unit, it
follows from Claim 1 that
\begin{align*}
&[\Phi(I)\bullet (i(P_{2}-P_{1})),A_{12}+B_{21}]_{*}+[I\bullet \Phi(i(P_{2}-P_{1})),A_{12}+B_{21}]_{*}\\
&+[I\bullet (i(P_{2}-P_{1})),\Phi(A_{12}+B_{21})]_{*}\\
&=\Phi([I\bullet (i(P_{2}-P_{1})),A_{12}+B_{21}]_{*})\\
&=\Phi([I\bullet (i(P_{2}-P_{1})),A_{12}]_{*})+\Phi([I\bullet (i(P_{2}-P_{1})), B_{21}]_{*})\\
&=[\Phi(I)\bullet (i(P_{2}-P_{1})),A_{12}+B_{21}]_{*}+[I\bullet \Phi(i(P_{2}-P_{1})),A_{12}+B_{21}]_{*}\\
&+[I\bullet (i(P_{2}-P_{1})),\Phi(A_{12})+\Phi(B_{21})]_{*}.
\end{align*}
From this, we get $[I\bullet (i(P_{2}-P_{1})),T]_{*}=0.$ So
$T_{11}=T_{22}=0.$

Since $[I\bullet A_{12},P_{1}]_{*}=0$, it
follows that
\begin{align*}
&[\Phi(I)\bullet (A_{12}+B_{21}),P_{1}]_{*}+[I\bullet \Phi(A_{12}+B_{21}),P_{1}]_{*}\\
&+[I\bullet (A_{12}+B_{21}),\Phi(P_{1})]_{*}\\
&=\Phi([I\bullet (A_{12}+B_{21}),P_{1}]_{*})\\
&=\Phi([I\bullet A_{12},P_{1}]_{*})+\Phi([I\bullet B_{21},P_{1}]_{*})\\
&=[\Phi(I)\bullet (A_{12}+B_{21}),P_{1}]_{*}+[I\bullet (\Phi(A_{12})+\Phi(B_{21})),P_{1}]_{*}\\
&+[I\bullet (A_{12}+B_{21}),\Phi(P_{1})]_{*}.
\end{align*}
Hence $[I\bullet T,P_{1}]_{*}=0$, from which we get that $T_{21}=0.$ Similarly, we can prove that $T_{12}=0,$  proving the step.\\
\textbf{Step 2.2}. For every
$A_{11}\in\mathcal{A}_{11},B_{12}\in\mathcal {A}_{12},
C_{21}\in\mathcal {A}_{21}, D_{22}\in\mathcal {A}_{22},$ we have
$$\Phi(A_{11}+B_{12}+C_{21})=\Phi(A_{11})+\Phi(B_{12})+\Phi(C_{21})$$
and
$$\Phi(B_{12}+C_{21}+D_{22})=\Phi(B_{12})+\Phi(C_{21})+\Phi(D_{22}).$$

Let $T=\Phi(A_{11}+B_{12}+C_{21})-\Phi(A_{11})-\Phi(B_{12})-\Phi(C_{21})$.

It follows from Step 2.1 that
that
\begin{align*}
&[\Phi(I)\bullet (iP_{2}),A_{11}+B_{12}+C_{21}]_{*}+[I\bullet \Phi(iP_{2}),A_{11}+B_{12}+C_{21}]_{*}\\
&+[I\bullet (iP_{2}),\Phi(A_{11}+B_{12}+C_{21})]_{*}\\
&=\Phi([I\bullet (iP_{2}),A_{11}+B_{12}+C_{21}]_{*})\\
&=\Phi([I\bullet (iP_{2}),A_{11}]_{*})+\Phi([I\bullet (iP_{2}), B_{12}+C_{21}]_{*})\\
&=\Phi([I\bullet (iP_{2}),A_{11}]_{*})+\Phi([I\bullet (iP_{2}), B_{12}]_{*})+\Phi([I\bullet (iP_{2}), C_{21}]_{*})\\
&=[\Phi(I)\bullet (iP_{2}),A_{11}+B_{12}+C_{21}]_{*}+[I\bullet \Phi(iP_{2}),A_{11}+B_{12}+C_{21}]_{*}\\
&+[I\bullet (iP_{2}),\Phi(A_{11})+\Phi(B_{12})+\Phi(C_{21})]_{*}.
\end{align*}
From this, we get $[I\bullet (iP_{2}),T]_{*}=0.$ So
$T_{12}=T_{21}=T_{22}=0.$

Since $$[I\bullet(i(P_{2}-P_{1})),B_{12}]_{*}=[I\bullet(i(P_{2}-P_{1})),C_{21}]_{*}=0,$$ it
follows  that \begin{align*}
&[\Phi(I)\bullet (i(P_{2}-P_{1})),A_{11}+B_{12}+C_{21}]_{*}+[I\bullet \Phi(i(P_{2}-P_{1})),A_{11}+B_{12}+C_{21}]_{*}\\
&+[I\bullet (i(P_{2}-P_{1})),\Phi(A_{11}+B_{12}+C_{21})]_{*}\\
&=\Phi([I\bullet (i(P_{2}-P_{1})),A_{11}+B_{12}+C_{21}]_{*})\\
&=\Phi([I\bullet (i(P_{2}-P_{1})),A_{11}]_{*})+\Phi([I\bullet (i(P_{2}-P_{1})), B_{12}]_{*})+\Phi([I\bullet (i(P_{2}-P_{1})), C_{21}]_{*})\\
&=[\Phi(I)\bullet (i(P_{2}-P_{1})),A_{11}+B_{12}+C_{21}]_{*}+[I\bullet \Phi(i(P_{2}-P_{1})),A_{11}+B_{12}+C_{21}]_{*}\\
&+[I\bullet (i(P_{2}-P_{1})),\Phi(A_{11})+\Phi(B_{12})+\Phi(C_{21})]_{*},
\end{align*}
from which we get $[I\bullet (i(P_{2}-P_{1})),T]_{*}=0.$  So
$T_{11}=0,$  and then $T=0$.
Similarly, we can get that
$\Phi(B_{12}+C_{21}+D_{22})=\Phi(B_{12})+\Phi(C_{21})+\Phi(D_{22}).$\\\\
\textbf{Step 2.3}. For every $A_{11}\in\mathcal {A}_{11},
B_{12}\in\mathcal {A}_{12}, C_{21}\in\mathcal {A}_{21},
D_{22}\in\mathcal {A}_{22} ,$ we have
$$\Phi(A_{11}+B_{12}+C_{21}+D_{22})=\Phi(A_{11})+\Phi(B_{12})+\Phi(C_{21})+\Phi(D_{22}).$$

Let $T=\Phi(A_{11}+B_{12}+C_{21}+D_{22})-\Phi(A_{11})-\Phi(B_{12})-\Phi(C_{21})-\Phi(D_{22}).$
It follows from Step 2.2 that
\begin{align*}
&[\Phi(I)\bullet (iP_{2}),A_{11}+B_{12}+C_{21}+D_{22}]_{*}+[I\bullet \Phi(iP_{2}),A_{11}+B_{12}+C_{21}+D_{22}]_{*}\\
&+[I\bullet (iP_{2}),\Phi(A_{11}+B_{12}+C_{21}+D_{22})]_{*}\\
&=\Phi([I\bullet (iP_{2}),A_{11}+B_{12}+C_{21}+D_{22}]_{*})\\
&=\Phi([I\bullet (iP_{2}),A_{11}]_{*})+\Phi([I\bullet (iP_{2}), B_{12}+C_{21}+D_{22})\\
&=\Phi([I\bullet (iP_{2}),A_{11}]_{*})+\Phi([I\bullet (iP_{2}), B_{12}]_{*})+\Phi([I\bullet (iP_{2}), C_{21}]_{*})+\Phi([I\bullet (iP_{2}), D_{22}]_{*})\\
&=[\Phi(I)\bullet (iP_{2}),A_{11}+B_{12}+C_{21}++D_{22}]_{*}+[I\bullet \Phi(iP_{2}),A_{11}+B_{12}+C_{21}+D_{22}]_{*}\\
&+[I\bullet (iP_{2}),\Phi(A_{11})+\Phi(B_{12})+\Phi(C_{21})+\Phi(D_{22})]_{*}.
\end{align*}
From this, we get $[I\bullet (iP_{2}),T]_{*}=0.$ So
$T_{12}=T_{21}=T_{22}=0.$
Similarly,  we can prove $T_{11}=0,$ proving the step.\\
\textbf{Step 2.4}. For every $A_{jk}, B_{jk}\in\mathcal {A}_{jk}, 1\leq j\neq k\leq 2,$ we have
$$\Phi(A_{jk}+B_{jk})=\Phi(A_{jk})+\Phi(B_{jk}).$$

Since
$$[\frac{I}{2}\bullet (P_{j}+A_{jk}),P_{k}+B_{jk}]_{*}=(A_{jk}+B_{jk})-A^{\ast}_{jk}-B_{jk}A^{\ast}_{jk},$$ we get from
Step 2.3 that
\begin{align*}&\Phi(A_{jk}+B_{jk})+\Phi(-A^{\ast}_{jk})+\Phi(-B_{jk}A^{\ast}_{jk})\\
&=\Phi([\frac{I}{2}\bullet (P_{j}+A_{jk}),P_{k}+B_{jk}]_{*})\\
&=[\Phi(\frac{I}{2})\bullet (P_{j}+A_{jk}),P_{k}+B_{jk}]_{*}+[\frac{I}{2}\bullet \Phi(P_{j}+A_{jk}),P_{k}+B_{jk}]_{*}\\
&+[\frac{I}{2}\bullet (P_{j}+A_{jk}),\Phi(P_{k}+B_{jk})]_{*}\\
&=[\Phi(\frac{I}{2})\bullet (P_{j}+A_{jk}),P_{k}+B_{jk}]_{*}+[\frac{I}{2}\bullet (\Phi(P_{j})+\Phi(A_{jk})),P_{k}+B_{jk}]_{*}\\
&+[\frac{I}{2}\bullet (P_{j}+A_{jk}),(\Phi(P_{k})+\Phi(B_{jk}))]_{*}\\
&=\Phi([\frac{I}{2}\bullet P_{j},P_{k}]_{*})+\Phi([\frac{I}{2}\bullet P_{j},B_{jk}]_{*})+\Phi([\frac{I}{2}\bullet A_{jk},P_{k}]_{*})+\Phi([\frac{I}{2}\bullet A_{jk},B_{jk}]_{*})\\
&=\Phi(B_{jk})+\Phi(A_{jk}-A^{\ast}_{jk})+\Phi(-B_{jk}A^{\ast}_{jk})\\
&=\Phi(B_{jk})+\Phi(A_{jk})+\Phi(-A^{\ast}_{jk})+\Phi(-B_{jk}A^{\ast}_{jk}).
\end{align*}
 Then
$$\Phi(A_{jk}+B_{jk})=\Phi(A_{jk})+\Phi(B_{jk}).$$\\
\textbf{Step 2.5}. For every $A_{jj}, B_{jj}\in\mathcal {A}_{jj}, 1\leq j\leq 2,$ we have
$$\Phi(A_{jj}+B_{jj})=\Phi(A_{jj})+\Phi(B_{jj}).$$

Let $T=\Phi(A_{jj}+B_{jj})-\Phi(A_{jj})-\Phi(B_{jj}).$   For $1\leq j\neq k\leq 2$, it
follows  that
\begin{align*}
&[\Phi(I)\bullet (iP_{k}),A_{jj}+B_{jj}]_{*}+[I\bullet \Phi(iP_{k}),A_{jj}+B_{jj}]_{*}\\
&+[I\bullet (iP_{k}),\Phi(A_{jj}+B_{jj})]_{*}\\
&=\Phi([I\bullet (iP_{k}),A_{jj}+B_{jj}]_{*})\\
&=\Phi([I\bullet (iP_{k}),A_{jj}]_{*})+\Phi([I\bullet (iP_{k}), B_{jj}]_{*})\\
&=[\Phi(I)\bullet (iP_{k}),A_{jj}+B_{jj}]_{*}+[I\bullet \Phi(iP_{k}),A_{jj}+B_{jj}]_{*}\\
&+[I\bullet (iP_{k}),\Phi(A_{jj})+\Phi(B_{jj})]_{*}.
\end{align*}
From this, we get $[I\bullet (iP_{k}),T]_{*}=0.$ So
$T_{jk}=T_{kj}=T_{kk}=0.$  Now we get $T=T_{jj}.$

For every $C_{jk}\in\mathcal {A}_{jk}, j\neq k$,
it follows from Step 2.4 that
\begin{align*}
&[\Phi(I)\bullet (A_{jj}+B_{jj}),C_{jk}]_{*}+[I\bullet \Phi(A_{jj}+B_{jj}),C_{jk}]_{*}\\
&+[I\bullet (A_{jj}+B_{jj}),\Phi(C_{jk})]_{*}\\
&=\Phi([I\bullet (A_{jj}+B_{jj}),C_{jk}]_{*})\\
&=\Phi([I\bullet A_{jj},C_{jk}]_{*})+\Phi([I\bullet B_{jj},C_{jk}]_{*})\\
&=[\Phi(I)\bullet (A_{jj}+B_{jj}),C_{jk}]_{*}+[I\bullet \Phi(A_{jj})+\Phi(B_{jj}),C_{jk}]_{*}\\
&+[I\bullet (A_{jj}+B_{jj}),\Phi(C_{jk})]_{*}.
\end{align*}
Hence $[I\bullet T_{jj},C_{jk}]_{*}=0$ for all $C_{jk}\in\mathcal {A}_{jk},$ that is, $T_{jj}CP_{k}=0$ for all $C\in\mathcal {A}.$
It follows from  $(\spadesuit)$ and $(\clubsuit)$ that $T=T_{jj}=0,$ proving the step.

Now,  it follows from  Steps 2.3, 2.4 and 2.5 that $\Phi$ is additive, proving the Claim 2.\\
\textbf{Claim 3}.   $\Phi(I)$ is a self-adjont central element in $\mathcal {A}$.

On the one hand, we have
\begin{align*}0&=\Phi([I\bullet I, I]_{*})\\
&=[\Phi(I)\bullet I, I]_{*}+[I\bullet \Phi(I), I]_{*}+[I\bullet I, \Phi(I)]_{*}\\
&=[2\Phi(I), I]_{*}\\
&=2\Phi(I)-2\Phi(I)^{*},\end{align*}
which implies that  $\Phi(I)$ is a self-adjont element in $\mathcal {A}$.

On the other hand, for all $A\in\mathcal {A},$ we get
\begin{align*}0&=\Phi([I\bullet I, A]_{*})\\
&=[\Phi(I)\bullet I, A]_{*}+[I\bullet \Phi(I), A]_{*}+[I\bullet I, \Phi(A)]_{*}\\
&=2[2\Phi(I), A]_{*}\\
&=4(\Phi(I)A-A\Phi(I)),\end{align*}
which implies that  $\Phi(I)$ is a central element in $\mathcal {A}$.\\
\textbf{Claim 4}. $P_{1}\Phi(P_{1})P_{2}=-P_{1}\Phi(P_{2})P_{2}, P_{2}\Phi(P_{1})P_{1}=-P_{2}\Phi(P_{2})P_{1},
P_{1}\Phi(P_{2})P_{1}=P_{2}\Phi(P_{1})P_{2}=0.$

On the one hand, for $1\leq j\neq k\leq2$, it follows from Claim 3 that
\begin{align*}0&=\Phi([I\bullet P_{j}, P_{k}]_{*})\\
&=[\Phi(I)\bullet P_{j}, P_{k}]_{*}+[I\bullet \Phi(P_{j}), P_{k}]_{*}+[I\bullet P_{j}, \Phi(P_{k})]_{*}\\
&=[2\Phi(P_{j}), P_{k}]_{*}+[2P_{j}, \Phi(P_{k})]_{*}\\
&=2\Phi(P_{j})P_{k}-2P_{k}\Phi(P_{j})^{*}+2P_{j}\Phi(P_{k})-2\Phi(P_{k})P_{j}.\end{align*}
Multiplying both sides of the above equation by $P_{j}$ and $P_{k}$ from the left and right respectively, we obtain that
$P_{1}\Phi(P_{1})P_{2}=-P_{1}\Phi(P_{2})P_{2}$ and $P_{2}\Phi(P_{1})P_{1}=-P_{2}\Phi(P_{2})P_{1}$.

On the other hand,  we get
\begin{align*}0&=\Phi([I\bullet (iP_{j}), P_{k}]_{*})\\
&=[\Phi(I)\bullet (iP_{j}), P_{k}]_{*}+[I\bullet \Phi(iP_{j}), P_{k}]_{*}+[I\bullet (iP_{j}), \Phi(P_{k})]_{*}\\
&=[2\Phi(iP_{j}), P_{k}]_{*}+[2iP_{j}, \Phi(P_{k})]_{*}\\
&=2\Phi(iP_{j})P_{k}-2P_{k}\Phi(iP_{j})^{*}+2i(P_{j}\Phi(P_{k})+\Phi(P_{k})P_{j}).\end{align*}
Multiplying both sides of the above equation by $P_{j}$ from the left and right respectively, we obtain that
$P_{1}\Phi(P_{2})P_{1}=P_{2}\Phi(P_{1})P_{2}=0.$\\
\textbf{Claim 5}. $P_{1}\Phi(P_{1})P_{1}=P_{2}\Phi(P_{2})P_{2}=0.$

For every $A_{12}\in\mathcal {A}_{12}$, on the one hand, it follows from Claim 2 and Claim 3 that
\begin{align*}2\Phi(A_{12})&=\Phi([I\bullet P_{1}, A_{12}]_{*})\\
&=[\Phi(I)\bullet P_{1}, A_{12}]_{*}+[I\bullet \Phi(P_{1}), A_{12}]_{*}+[I\bullet P_{1}, \Phi(A_{12})]_{*}\\
&=[2\Phi(I) P_{1}, A_{12}]_{*}+[2\Phi(P_{1}), A_{12}]_{*}+[2P_{1}, \Phi(A_{12})]_{*}\\
&=2\Phi(I)A_{12}+2\Phi(P_{1})A_{12}-2A_{12}\Phi(P_{1})^{*}+2P_{1}\Phi(A_{12})-2\Phi(A_{12})P_{1}.\end{align*}
Multiplying both sides of the above equation by $P_{1}$ and $P_{2}$ from the left and right respectively, by Claim 4, we get that
\begin{equation}\label{eq 2.1}
P_{1}\Phi(P_{1})A_{12}+\Phi(I)A_{12}=0.
\end{equation}

On the other hand, we have
\begin{align*}2\Phi(A_{12})&=\Phi([P_{1}\bullet P_{1}, A_{12}]_{*})\\
&=[\Phi(P_{1})\bullet P_{1}, A_{12}]_{*}+[P_{1}\bullet \Phi(P_{1}), A_{12}]_{*}+[P_{1}\bullet P_{1}, \Phi(A_{12})]_{*}\\
&=[\Phi(P_{1})P_{1}+P_{1}\Phi(P_{1})^{*}, A_{12}]_{*}+[P_{1}\Phi(P_{1})+\Phi(P_{1})P_{1}, A_{12}]_{*}+[2P_{1}, \Phi(A_{12})]_{*}\\
&=\Phi(P_{1})A_{12}+P_{1}\Phi(P_{1})^{*}A_{12}-A_{12}\Phi(P_{1})P_{1}+P_{1}\Phi(P_{1})A_{12}\\
&+\Phi(P_{1})A_{12}-A_{12}\Phi(P_{1})^{*}P_{1}
+2P_{1}\Phi(A_{12})-2\Phi(A_{12})P_{1}.\end{align*}
Multiplying both sides of the above equation by $P_{1}$ and $P_{2}$ from the left and right respectively, we get that
\begin{equation} \label{eq 2.2} 3P_{1}\Phi(P_{1})A_{12}+P_{1}\Phi(P_{1})^{*}A_{12}=0.\end{equation}

Finally, we also have that
\begin{align*}2\Phi(A_{12})&=\Phi([P_{1}\bullet I, A_{12}]_{*})\\
&=[\Phi(P_{1})\bullet I, A_{12}]_{*}+[P_{1}\bullet \Phi(I), A_{12}]_{*}+[P_{1}\bullet I, \Phi(A_{12})]_{*}\\
&=[\Phi(P_{1})+\Phi(P_{1})^{*}, A_{12}]_{*}+[2P_{1}\Phi(I), A_{12}]_{*}+[2P_{1}, \Phi(A_{12})]_{*}\\
&=(\Phi(P_{1})+\Phi(P_{1})^{*})A_{12}-A_{12}(\Phi(P_{1})+\Phi(P_{1})^{*})+2\Phi(I)A_{12}+2P_{1}\Phi(A_{12})-2\Phi(A_{12})P_{1}.\end{align*}
Multiplying both sides of the above equation by $P_{1}$ and $P_{2}$ from the left and right respectively, by Claim 4,  we get that
\begin{equation}\label{eq 2.3} P_{1}\Phi(P_{1})A_{12}+P_{1}\Phi(P_{1})^{*}A_{12}+2\Phi(I)A_{12}=0.\end{equation}
It follows from  Eq. (\ref{eq 2.2}) and Eq. (\ref{eq 2.3}) that
\begin{equation}\label{eq 2.4}
P_{1}\Phi(P_{1})A_{12}-\Phi(I)A_{12}=0.
\end{equation}
Now, by  Eq. (\ref{eq 2.1}) and Eq. (\ref{eq 2.4}), we have $P_{1}\Phi(P_{1})A_{12}=0$, that is $P_{1}\Phi(P_{1})P_{1}AP_{2}=0$
for all $A\in\mathcal {A}$. It follows from  $(\clubsuit)$ that $P_{1}\Phi(P_{1})P_{1}=0.$ Similarly,  we can prove $P_{2}\Phi(P_{2})P_{2}=0.$\\
\textbf{Claim 6}. $\Phi(I)=0.$

By Claims 2, 4 and 5, we can get that
$$\Phi(I)=\Phi(P_{1})+\Phi(P_{2})=P_{1}\Phi(P_{1})P_{2}+P_{2}\Phi(P_{1})P_{1}+P_{1}\Phi(P_{2})P_{2}+P_{2}\Phi(P_{2})P_{1}=0.$$
\textbf{Claim 7}. For all $A, B\in\mathcal {A}$, we have $\Phi([A,B]_{*})=[\Phi(A),B]_{*}+[A, \Phi(B)]_{*}$.

It follows from  Claim 2 and Claim 6 that \begin{align*}2\Phi([A,B]_{*})&=\Phi([I\bullet A, B]_{*})\\
&=[\Phi(I)\bullet A, B]_{*}+[I\bullet \Phi(A), B]_{*}+[I\bullet A, \Phi(B)]_{*}\\
&=[2\Phi(A), B]_{*}+[2A, \Phi(B)]_{*}\\
&=2([\Phi(A),B]_{*}+[A, \Phi(B)]_{*}),\end{align*}
which implies that  $\Phi([A,B]_{*})=[\Phi(A),B]_{*}+[A, \Phi(B)]_{*}$.\\
\textbf{Claim 8}. For all $A\in\mathcal {A}$, $\Phi(A^{*})=\Phi(A)^{*}.$

For every $A\in\mathcal {A}$, by Claims 2, 6 and 7, we have
$$
\Phi(A)-\Phi(A^{*})
=\Phi([A,I]_{*})
=[\Phi(A), I]_{*}
=\Phi(A)-\Phi(A)^{*}.
$$
Hence
$\Phi(A^{*})=\Phi(A)^{*}.$\\

Now, let $T=P_{1}\Phi(P_{1})P_{2}-P_{2}\Phi(P_{1})P_{1}.$ Then $T^{*}=-T$.
Defining a map $\phi: \mathcal {A}\rightarrow \mathcal {A}$
by $\phi(A)=\Phi(A)-(AT-TA)$ for all $A\in\mathcal {A}.$ It is easy to verify that
$\phi$ has the following properties.\\
\textbf{Claim 9}. \begin{itemize} \item[(1)] For all $A, B\in\mathcal {A}$,  $\phi([A,B]_{*})=[\phi(A),B]_{*}+[A, \phi(B)]_{*};$

\item[(2)] $\phi(P_{i})=0, i=1,2;$

\item[(3)] $\phi$ is additive;

\item[(4)] $\phi$ is an additive derivation if and only if $\Phi$ is an additive derivation.
\end{itemize}
\textbf{Claim 10}. $\phi(\mathcal {A}_{ij})\subseteq\mathcal
{A}_{ij}, i,j=1,2.$

Let $A_{ij}\in\mathcal {A}_{ij}, 1\leq i\neq j\leq2.$ On the one
hand, it follows from $\phi(P_{i})=0$ that
$$\phi(A_{ij})=\phi([P_{i}, A_{ij}]_{*})
=[P_{i},\phi(A_{ij})]_{*}=P_{i}\phi(A_{ij})-\phi(A_{ij})P_{i}.
$$
Hence $P_{i}\phi(A_{ij})P_{i}=P_{j}\phi(A_{ij})P_{j}=P_{j}\phi(A_{ij})P_{i}=0.$ So $\phi({A}_{ij})=P_{i}\phi(
{A}_{ij})P_{j}\in \mathcal {A}_{ij}, 1\leq i\neq j\leq2.$

Let $A_{ii}\in\mathcal {A}_{ii}, i=1,2.$  Then
$$0=\phi([P_{j}, A_{ii}]_{*})
= [P_{j}, \phi(A_{ii})]_{*}=P_{j}\phi(A_{ii})-\phi(A_{ii})P_{j}.$$ Hence
$P_{i}\phi(A_{ii})P_{j}=P_{j}\phi(A_{ii})P_{i}=0.$ Now we let
$\phi(A_{ii})=P_{i}\phi(A_{ii})P_{i}+P_{j}\phi(A_{ii})P_{j}.$ For
any $B_{ij}\in\mathcal {A}_{ij},1\leq i\neq j\leq2,$ it follows
from $\phi(B_{ij})\in\mathcal {A}_{ij}$ that
$$0=\phi([B_{ij},A_{ii}]_{*})=[\phi(B_{ij}), A_{ii}]_{*}
+[B_{ij}, \phi(A_{ii})]_{*}=B_{ij}P_{j}\phi(A_{ii})P_{j}-P_{j}\phi(A_{ii})P_{j}B_{ij}^{*}.$$
So $P_{j}\phi(A_{ii})P_{j}B_{ij}^{*}=0$, that is
$P_{j}\phi(A_{ii})P_{j}BP_{i}=0$
holds true for any $B\in\mathcal {A}.$
It follows from $(\spadesuit)$ and $(\clubsuit)$ that
$P_{j}\phi(A_{ii})P_{j}=0.$ Now we get $\phi( {A}_{ii})=P_{i}\phi(
{A}_{ii})P_{i}\in \mathcal {A}_{ii}, i=1,2.$ \\
\textbf{Claim 11}. Let $A_{ii}, B_{ii}\in\mathcal {A}_{ii}$ and
$A_{ij}, B_{ij}\in\mathcal {A}_{ij}, 1\leq i\neq j\leq2.$ Then
$$\phi(A_{ii}B_{ii})=\phi(A_{ii})B_{ii}+A_{ii}\phi(B_{ii}),\phi(A_{ii}B_{ij})=\phi(A_{ii})B_{ij}+A_{ii}\phi(B_{ij}),$$
$$\phi(A_{ij}B_{ji})=\phi(A_{ij})B_{ji}+A_{ij}\phi(B_{ji}),\phi(A_{ij}B_{jj})=\phi(A_{ij})B_{jj}+A_{ij}\phi(B_{jj}).$$

It follows from Claim 10 that
\begin{align*}\phi(A_{ii}B_{ij})&=\phi([A_{ii}, B_{ij}]_{*})
=[\phi(A_{ii}), B_{ij}]_{*}+[A_{ii}, \phi(B_{ij}]_{*})
\\&=\phi(A_{ii})B_{ij}+A_{ii}\phi(B_{ij}).\end{align*}

For any $C_{ij}\in\mathcal {A}_{ij},$ we have
\begin{align*}\phi(A_{ii}B_{ii})C_{ij}+A_{ii}B_{ii}\phi(C_{ij})&=\phi(A_{ii}B_{ii}C_{ij})
\\&=\phi(A_{ii})B_{ii}C_{ij}+A_{ii}\phi(B_{ii}C_{ij})
\\&=\phi(A_{ii})B_{ii}C_{ij}+A_{ii}\phi(B_{ii})C_{ij}+A_{ii}B_{ii}\phi(C_{ij}).\end{align*}
Then $(\phi(A_{ii}B_{ii})-\phi(A_{ii})B_{ii}-A_{ii}\phi(B_{ii}))C_{ij}=0$
for any $C_{ij}\in\mathcal {A}_{ij}.$ It follows from  $(\spadesuit)$ and $(\clubsuit)$  that
$\Phi(A_{ii}B_{ii})=\Phi(A_{ii})B_{ii}+A_{ii}\Phi(B_{ii}).$

It follows from
Claim 10 that
\begin{align*}\phi(A_{ij}B_{ji})&=\phi([A_{ij}, B_{ji}]_{*})\\&=[\phi( A_{ij}),B_{ji}]_{*}+ [A_{ij}, \phi(B_{ji})]_{*}\\&=\phi(A_{ij})B_{ji}+A_{ij}\phi(B_{ji}).\end{align*}

For any $C_{ji}\in\mathcal {A}_{ji}$, we have
\begin{align*}\phi(C_{ji})A_{ij}B_{jj}+C_{ji}\phi(A_{ij}B_{jj})&=\phi(C_{ji}A_{ij}B_{jj})\\&=\phi(C_{ji}A_{ij})B_{jj}
+C_{ji}A_{ij}\Phi(B_{jj})\\&=\phi(C_{ji})A_{ij}B_{jj}+C_{ji}\phi(A_{ij})B_{jj}+C_{ji}A_{ij}\phi(B_{jj})
.\end{align*} Hence
$C_{ji}(\phi(A_{ij}B_{jj})-\phi(A_{ij})B_{jj}-A_{ij}\phi(B_{jj}))=0$
for any $C_{ji}\in\mathcal {A}_{ji}.$ It follows from  $(\spadesuit)$ and $(\clubsuit)$  that
$\phi(A_{ij}B_{jj})=\phi(A_{ij})B_{jj}+A_{ij}\phi(B_{jj}).$  \\
\textbf{Claim 12}.  $\Phi(AB)=\Phi(A)B+A\Phi(B)$ for all $A, B\in\mathcal {A}.$

Write $A=\sum^{2}_{i,j=1}A_{ij}, B=\sum^{2}_{i,j=1}B_{ij}\in\mathcal {A}$. Then $AB=A_{11}B_{11}+A_{11}B_{12}+A_{12}B_{21}+A_{12}B_{22}+A_{21}B_{11}+A_{21}B_{12}+A_{22}B_{21}+A_{22}B_{22}$.
It follows from Claim 11 and the additivity of $\phi$ that $\phi(AB)=\phi(A)B+A\phi(B).$ So  $\Phi(AB)=\Phi(A)B+A\Phi(B).$

Now, by Claims 2, 8 and 12, we have proved that $\Phi$ is an additive
$*-$derivation. This completes the proof of Theorem 2.1.
\end{proof}

\section{Corollaries}

In this section, we present some corollaries of  the main result.
An algebra $\mathcal{A}$ is called prime if $A\mathcal
{A}B=\{0\}$ for $A, B\in\mathcal {A}$ implies either $A=0$ or $B=0$.
Observing that prime $*$-algebras satisfy $(\spadesuit)$ and $(\clubsuit)$, we have the following corollary.

\begin{corollary}
Let $\mathcal {A}$ be a prime $*$-algebra with unit $I$ and $P$ be a nontrivial projection in $\mathcal {A}$. Then $\Phi$ is a nonlinear  mixed Jordan triple $*$-derivation on $\mathcal {A}$ if and only if $\Phi$ is an additive $*-$derivation.
\end{corollary}

A von Neumann algebra $\mathcal{M}$ is a weakly closed, self-adjoint algebra of operators on a Hilbert space $\mathcal{H}$ containing the identity operator $I$. It is shown in \cite{4} and \cite{Li} that if a von Neumann algebra has no central summands of type $I_1$, then $\mathcal{M}$ satifies $(\spadesuit)$ and $(\clubsuit)$. Now we have the following corollary.

\begin{corollary}
Let $\mathcal{M}$ be a von Neumann algebra with no central summands of type $I_1$. Then $\Phi: \mathcal{M} \rightarrow \mathcal{M}$ is a nonlinear  mixed Jordan triple $*$-derivation if and only if $\Phi$ is an additive $*-$derivation.
\end{corollary}

$\mathcal{M}$ is a factor von Neumann algebra if its center only contains the scalar operators. It is well known that a factor von Neumann algebra is prime and then we have the following corollary.

\begin{corollary}
Let $\mathcal{M}$ be a factor von Neumann algebra with dim$\mathcal{M}\geq2$. Then $\Phi: \mathcal{M} \rightarrow \mathcal{M}$ is a nonlinear  mixed Jordan triple $*$-derivation if and only if $\Phi$ is an additive $*-$derivation.
\end{corollary}

 $B(\mathcal{H})$ denotes the algebra of all bounded linear operators
on a complex Hilbert space $\mathcal{H}$. We denote the subalgebra of all bounded
finite rank operators by $\mathcal{F}(H)\subseteq B(\mathcal{H})$. We call a subalgebra $\mathcal {A}$ of $B(\mathcal{H})$ a
standard operator algebra if it contains $\mathcal{F}(H)$. Now we have the following corollary.

\begin{corollary}
Let $\mathcal{H}$ be an infinite dimensional complex Hilbert space
and $\mathcal {A}$ be a standard operator algebra on $\mathcal{H}$ containing the identity operator $I$.
Suppose that $\mathcal {A}$ is closed under the adjoint operation. Then $\Phi: \mathcal{A} \rightarrow \mathcal{A}$ is
a nonlinear  mixed Jordan triple $*$-derivation if and only if $\Phi$ is a linear $*$-derivation.
Moreover, there exists an operator $T\in B(\mathcal{H})$ satisfying $T + T^{*}= 0$ such that
$\Phi(A)=AT-TA$
for all $A\in A$, i.e., $\Phi$ is inner.
\end{corollary}
\begin{proof}
Since $\mathcal {A}$ is prime, we have that $\Phi$ is an additive $*$-derivation.
It follows from
\cite{p} that $\Phi$ is a linear
inner derivation, i.e., there exists an operator $S\in B(\mathcal{H})$ such that
$\Phi(A)=AS-SA$.
Using the fact $\Phi(A^{*}) =\Phi(A)^{*}$, we have
$$A^{*}S-SA^{*}=\Phi(A^{*}) = \Phi(A)^{*}=-A^{*}S^{*} + S^{*}A^{*}$$
for all $A\in A$. This leads to $A^{*}(S + S^{*})=(S + S^{*})A^{*}$. Hence, $S + S^{*} = \lambda I$
for some $\lambda\in \mathbb{R}$. Let us set $T =S-\frac{1}{2}\lambda I$. One can check that $T + T^{*}= 0$ such that
$\Phi(A)=AT-TA$.
\end{proof}

\bibliographystyle{amsplain}

\end{document}